\documentclass[a4paper]{article}
\usepackage{fullpage}
\usepackage{amsmath}
\usepackage{amssymb}
\usepackage{amsthm}
\usepackage{graphicx}
\usepackage{bbm}
\usepackage{enumerate}
\usepackage{microtype}
\usepackage{xcolor}
\usepackage[none]{hyphenat}
\usepackage[colorlinks=true,urlcolor=blue,linkcolor=blue,plainpages=false,pdfpagelabels]{hyperref}
\allowdisplaybreaks

\newtheorem{theorem}{Theorem}[section]
\newtheorem{proposition}[theorem]{Proposition}
\newtheorem{lemma}[theorem]{Lemma}

\newtheorem{definition}[theorem]{Definition}

\newcommand{\N}{\mathbb{N}}

\newcommand{\eps}{\varepsilon}
\newcommand{\vp}{\varphi}

\newcommand{\wt}[1]{\widetilde{#1}}
\newcommand{\vv}[1]{\overrightarrow{#1}}

\DeclareMathOperator{\Fix}{Fix}

\title{Revisiting jointly firmly nonexpansive families of mappings}
\author{Andrei Sipo\c s${}^{a,b}$\\[2mm]
\footnotesize ${}^a$Research Center for Logic, Optimization and Security (LOS), Department of Computer Science,\\
\footnotesize Faculty of Mathematics and Computer Science, University of Bucharest,\\
\footnotesize Academiei 14, 010014 Bucharest, Romania\\[1mm]
\footnotesize ${}^b$Simion Stoilow Institute of Mathematics of the Romanian Academy,\\
\footnotesize Calea Grivi\c tei 21, 010702 Bucharest, Romania\\[2mm]
\footnotesize E-mail: andrei.sipos@fmi.unibuc.ro\\
}
\date{}

\begin{document}

\maketitle

\begin{abstract}
Recently, the author, together with L. Leu\c stean and A. Nicolae, introduced the notion of jointly firmly nonexpansive families of mappings in order to investigate in an abstract manner the convergence of proximal methods. Here, we further the study of this concept, by giving a characterization in terms of the classical resolvent identity, by improving on the rate of convergence previously obtained for the uniform case, and by giving a treatment of the asymptotic behaviour at infinity of such families.


\end{abstract}

\section{Introduction}

In the paper \cite{LeuNicSip18}, Leu\c stean, Nicolae and the author introduced a unifying framework for studying the proximal point algorithm, a fundamental tool of convex optimization going back to Martinet \cite{Mar70}, Rockafellar \cite{Roc76} and Br\'ezis and Lions \cite{BreLio78} (for a detailed history of the relevant proximal methods, see the introduction to \cite{LeuNicSip18}). Even though maximal monotone operators already provided such a unified view in the realm of Hilbert spaces, recent developments in optimization techniques in nonlinear generalizations thereof such as geodesic spaces of non-positive curvature \cite{Bac13,Bac14,BacRei14} demanded higher levels of abstraction. Specifically, what the aforementioned three authors did was (i) to introduce the notion of {\it jointly firmly nonexpansive families of mappings}; (ii) to show that all the possible variants of mappings (usually called ``resolvents'') that are involved in the construction of the corresponding proximal iterations fit the definition; and (iii) to prove that the weak convergence of the iteration can be derived just from this joint firm nonexpansiveness condition.

In addition, it is known that in the so-called uniform cases of the proximal point algorithm (uniformly monotone operators, uniformly convex functions) one has uniqueness and strong convergence to the optimizing point (zero and minimizer, respectively). This ties into the area of {\it proof mining}, an applied subfield of mathematical logic primarily developed in the last decades by Ulrich Kohlenbach and his collaborators (the standard introduction is \cite{Koh08}, while a recent survey is \cite{Koh19}) that aims to analyze proofs in mainstream mathematics using proof-theoretical tools. More precisely, results due to Kohlenbach \cite{Koh90} and Kohlenbach and Oliva \cite{KohOli03} show that usually such an uniqueness property may be made quantitative in such a way as to yield as a consequence a convergence rate for an asymptotically regular iteration. By appropriately tweaking the notion of asymptotic regularity, such a rate was extracted in \cite{LeuNicSip18} that was independent of the sequence of step-sizes used in the construction of the iterative sequence.

{\it In this paper, we continue this study of jointly firmly nonexpansive families of mappings by providing a conceptual characterization of them and by giving new applications.}

Specifically, after reviewing in Section~\ref{sec:prelim} the basic facts that we need about geodesic spaces and firm nonexpansiveness, we present in Section~\ref{sec:relid} a definition of joint firm nonexpansiveness that is, we hope, more flexible than the one in \cite{LeuNicSip18}, and its immediate consequences that we shall need. The main result of this section, Theorem~\ref{th-res-id}, shows that a family of mappings is jointly firmly nonexpansive if and only if it satisfies the resolvent identity and each member of it is nonexpansive. In Section~\ref{sec:uniform}, we show how a recent quantitative result due to Kohlenbach and Powell \cite{KohPowXX} may be used to improve the conditions under which a rate of convergence may be obtained in the uniform case of the proximal point algorithm. Finally, in Section~\ref{sec:as-beh} we show how with minimal boundedness assumptions one can show that for any point, the curve that is obtained by applying to it all mappings in a jointly firmly nonexpansive family strongly converges to the projection of that point onto the common fixed point set of the family. This latter result, Theorem~\ref{t2}, unifies a number of results pertaining to the convergence of approximating curves going back to the 1960s.

\section{Preliminaries}\label{sec:prelim}

We say that a metric space $(X,d)$ is {\it geodesic} if for any two points $x$, $y \in X$ there is a geodesic that joins them, i.e. a mapping $\gamma : [0,1] \to X$ such that $\gamma(0)=x$, $\gamma(1)=y$ and for any $t$, $t' \in [0,1]$ we have that
$$d(\gamma(t),\gamma(t')) = |t-t'| d(x,y).$$
Among geodesic spaces, a subclass that is usually considered (e.g. in convex optimization) to be the proper nonlinear analogue of Hilbert spaces is the class of CAT(0) spaces, introduced by A. Aleksandrov \cite{Ale51} and named as such by M. Gromov \cite{Gro87}, defined as those geodesic spaces $(X,d)$ such that for any geodesic $\gamma : [0,1] \to X$ and for any $z \in X$ and $t \in [0,1]$ we have that
$$d^2(z,\gamma(t)) \leq (1-t)d^2(z,\gamma(0)) + td^2(z,\gamma(1)) - t(1-t)d^2(\gamma(0),\gamma(1)).$$
Another well-known fact about CAT(0) spaces is that each such space $(X,d)$ is {\it uniquely geodesic} -- that is, for any $x$, $y \in X$ there is a unique such geodesic $\gamma : [0,1] \to X$ that joins them -- and in this context we shall denote, for any $t \in [0,1]$, the point $\gamma(t)$ by $(1-t)x + ty$.

In 2008, Berg and Nikolaev proved (see \cite[Proposition 14]{BerNik08}) that in any metric space $(X,d)$, the function $\langle\cdot,\cdot\rangle : X^2 \times X^2 \to \mathbb{R}$, defined, for any $x$, $y$, $u$, $v \in X$, by
$$\langle \vv{xy}, \vv{uv} \rangle := \frac12(d^2(x,v) + d^2(y,u) - d^2(x,u) -d^2(y,v))$$
(where an ordered pair of points $(p,z) \in X^2$ is denoted by $\vv{pz}$), called the {\it quasi-linearization function}, is the unique one such that, for any $x$, $y$, $u$, $v$, $w \in X$, we have that:
\begin{enumerate}[(i)]
\item $\langle\vv{xy},\vv{xy}\rangle = d^2(x,y)$;
\item $\langle\vv{xy},\vv{uv}\rangle = \langle\vv{uv},\vv{xy}\rangle$;
\item $\langle\vv{yx},\vv{uv}\rangle = -\langle\vv{xy},\vv{uv}\rangle$;
\item $\langle\vv{xy},\vv{uv}\rangle + \langle\vv{xy},\vv{vw}\rangle = \langle\vv{xy},\vv{uw}\rangle$.
\end{enumerate}
The inner product notation is justified by the fact that if $X$ is a (real) Hilbert space, for any $x$, $y$, $u$, $v \in X$,
\begin{equation}\label{eq-quasi-inner}
\langle\vv{xy},\vv{uv}\rangle = \langle x-y,u-v \rangle = \langle y-x,v-u \rangle.
\end{equation} 
The main result of \cite{BerNik08}, Theorem 1, characterized CAT(0) spaces as being exactly those geodesic spaces $(X,d)$ such that the corresponding Cauchy-Schwarz inequality is satisfied, i.e. for any $x$, $y$, $u$, $v \in X$,
\begin{equation}\label{CauchySchwartz}
\langle\vv{xy},\vv{uv}\rangle \leq d(x,y)d(u,v).
\end{equation}

Firmly nonexpansive mappings were first introduced by Browder \cite{Bro67} in the context of Hilbert spaces and then by Bruck \cite{Bru73} in the context of Banach spaces (this later definition was also studied, e.g., in \cite{Rei77}). The following generalization to geodesic spaces, inspired by the study of firmly nonexpansive mappings in the Hilbert ball \cite{GoeRei82, GoeRei84, ReiSha87, ReiSha90} was introduced in \cite{AriLeuLop14}.

\begin{definition}
Let $X$ be a CAT(0) space. A mapping $T : X \to X$ is called {\em firmly nonexpansive} if for any $x,y \in X$ and any $t \in [0,1]$ we have that
$$d(Tx,Ty) \leq d((1-t)x + tTx,(1-t)y+tTy).$$
\end{definition}

As mentioned in \cite{AriLopNic15} (see also \cite{KohLopNic17}), if $X$ is a CAT(0) space, every firmly nonexpansive mapping $T:X \to X$ satisfies the so-called {\em property $(P_2)$}, i.e. that for all $x,y \in X$,
$$2d^2(Tx,Ty) \leq d^2(x,Ty) + d^2(y,Tx) - d^2(x,Tx) - d^2(y,Ty),$$
or, using the quasi-linearization function,
\begin{equation}\label{eq-quasi-P2}
d^2(Tx,Ty) \leq \langle\vv{TxTy},\vv{xy}\rangle.
\end{equation}
If $X$ is a Hilbert space, property $(P_2)$ coincides with firm nonexpansiveness as \eqref{eq-quasi-P2} and \eqref{eq-quasi-inner} yield $\|Tx - Ty\|^2 \leq \langle Tx-Ty,x-y \rangle$, which is equivalent to it e.g. by \cite[Proposition 4.2]{BauCom17}. Moreover, from this formulation given by \eqref{eq-quasi-P2} one immediately obtains, using \eqref{CauchySchwartz}, that a self-mapping of a CAT(0) space satisfying property $(P_2)$ is nonexpansive.

Throughout the paper, for any self-mapping $T$ (of an arbitrary set), we denote the set of its fixed points by $\Fix(T)$. 

In the remainder of this section, we shall introduce some notions that are needed in Section~\ref{sec:as-beh}. If $(X,d)$ is a metric space and $(x_n)_{n \in \N}$ is a sequence in $X$, then $(x_n)$ is called {\it metastable} if for any $\eps>0$ and $g:\N\to\N$ there is an $N$ such that for all $i$, $j \in [N,N+g(N)]$, $d(x_i,x_j)\leq\eps$, and that a {\it rate of metastability} for $(x_n)$ is a function $\Psi: (0,\infty) \times \N^\N \to \N$ such that for any $\eps$ and $g$, $\Psi(\eps,g)$ gives an upper bound on the (smallest) corresponding $N$. It is immediate that this is just a reformulation -- actually identifiable in mathematical logic as the Herbrand normal form -- of the Cauchy property; the concept was independently rediscovered by Tao \cite{Tao08} (it was named as such under a suggestion of Jennifer Chayes) and used successfully by him in proving a convergence result for multiple ergodic averages \cite{Tao08A}. The significance of this concept resides in the fact that rates of convergence for iterative sequences which are commonly employed in nonlinear analysis and convex optimization may not be uniform or computable (see \cite{Neu15}) -- in this case, a rate of metastability as introduced above is the next best thing that can be obtained; and the research program of proof mining mentioned in the introduction has achieved non-trivial extraction of such rates from celebrated strong convergence proofs, see, e.g., \cite{Koh11, KohKou15, KohSipXX}.

For all $g: \N \to \N$, we define $\wt{g} : \N \to \N$, for all $n$, by $\wt{g}(n):=n+g(n)$. Also, for all $f:\N \to \N$ and all $n \in \N$, we denote by $f^{(n)}$ the $n$-fold composition of $f$ with itself. Note that for all $g$ and $n$, $\wt{g}^{(n)}(0)\leq\wt{g}^{(n+1)}(0)$.

The following proposition gives a uniform and computable rate of metastability for nondecreasing sequences of nonnegative reals bounded above by a fixed constant.

\begin{proposition}[Quantitative Monotone Convergence Principle, cf. {\cite{Tao08}}]\label{qmcp}
Let $b>0$ and $(a_n)$ be a nondecreasing sequence in $[0,b]$. Then for all $\eps>0$ and $g:\N\to\N$ there is an $N \leq \wt{g}^{\left(\left\lceil\frac b\eps\right\rceil \right)}(0)$ such that for all $i$, $j \in [N,N+g(N)]$, $|a_i-a_j|\leq\eps$.
\end{proposition}

\begin{proof}
Let $\eps>0$ and $g:\N \to \N$. Assume that the conclusion is false, hence in particular for all $i \leq \left\lceil\frac b\eps\right\rceil$, $a_{\wt{g}^{(i+1)}(0)} - a_{\wt{g}^{(i)}(0)} > \eps$. Then
$$b \geq a_{\wt{g}^{\left(\left\lceil\frac b\eps\right\rceil + 1\right)}(0)} \geq a_{\wt{g}^{\left(\left\lceil\frac b\eps\right\rceil + 1\right)}(0)} - a_0 = \sum_{i=0}^{\left\lceil\frac b\eps\right\rceil} a_{\wt{g}^{(i+1)}(0)} - a_{\wt{g}^{(i)}(0)} > \left\lceil\frac b\eps\right\rceil\cdot\eps \geq b,$$
a contradiction.
\end{proof}

\section{The relationship with the resolvent identity}\label{sec:relid}

Fix a CAT(0) space $X$ for the remainder of this paper. If $T$ and $U$ are self-mappings of $X$ and $\lambda$, $\mu>0$, we say that $T$ and $U$ are {\it $(\lambda,\mu)$-mutually firmly nonexpansive} if for all $x$, $y\in X$ and all $\alpha$, $\beta \in [0,1]$ such that 
$(1-\alpha)\lambda=(1-\beta)\mu$, one has that
$$ d(Tx,Uy) \leq d((1-\alpha)x+\alpha Tx,(1-\beta)y+\beta Uy).$$
If $(T_n)_{n \in \N}$ is a family of self-mappings of $X$ and $(\gamma_n)_{n\in\N} \subseteq (0,\infty)$, we say that $(T_n)$ is {\it jointly firmly nonexpansive} with respect to $(\gamma_n)$ if for all $n$, $m\in\N$, $T_n$ and $T_m$ are $(\gamma_n,\gamma_m)$-mutually firmly nonexpansive. In addition, if $(T_\gamma)_{\gamma>0}$ is a family of self-mappings of $X$, we say that it is plainly {\it jointly firmly nonexpansive} if for all $\lambda$, $\mu > 0$, $T_\lambda$ and $T_\mu$ are $(\lambda,\mu)$-mutually firmly nonexpansive. It is clear that a family $(T_\gamma)$ is jointly firmly nonexpansive if and only if for every $(\gamma_n)_{n\in\N} \subseteq (0,\infty)$, $(T_{\gamma_n})_{n \in \N}$ is jointly firmly nonexpansive with respect to $(\gamma_n)$. In \cite{LeuNicSip18} it was shown that examples of jointly firmly nonexpansive families of mappings are furnished by resolvent-type mappings used in convex optimization -- specifically, by:
\begin{itemize}
\item the family $(J_{\gamma f})_{\gamma>0}$, where $f$ is a proper convex lower semicontinous function on $X$ and one denotes for any such function $g$ its proximal mapping by $J_g$;
\item the family $(R_{T, \gamma})_{\gamma>0}$, where $T$ is a nonexpansive self-mapping of $X$ and one denotes, for any $\gamma > 0$, its resolvent of order $\gamma$ by $R_{T, \gamma}$;
\item (if $X$ is a Hilbert space) the family $(J_{\gamma A})_{\gamma>0}$, where $A$ is a maximally monotone operator on $X$ and one denotes for any such operator $B$ its resolvent by $J_B$.
\end{itemize}

Again, if $T$ and $U$ are self-mappings of $X$ and $\lambda$, $\mu>0$, one says that $T$ and $U$ are {\it $(\lambda,\mu)$-mutually $(P_2)$} if for all $x$, $y\in X$, 
$$\frac1{\mu}(d^2(Tx,Uy) + d^2(y,Uy) - d^2(y,Tx)) \leq \frac1{\lambda}(d^2(x,Uy) - d^2(x,Tx) - d^2(Tx,Uy)),$$
or, using the quasi-linearization function,
$$\frac1{\mu}\langle\vv{TxUy},\vv{yUy}\rangle\leq \frac1{\lambda}\langle\vv{TxUy},\vv{xTx}\rangle.$$

\begin{proposition}[{\cite[Corollary 3.11]{LeuNicSip18}}]\label{same-fixed}
Any two mutually $(P_2)$ self-mappings of $X$ have the same fixed points.
\end{proposition}

\begin{proposition}\label{t-eq-u}
Let $T$ and $U$ be self-mappings of $X$ and $\lambda>0$ such that $T$ and $U$ are $(\lambda,\lambda)$-mutually $(P_2)$. Then $T=U$.
\end{proposition}

\begin{proof}
Let $x \in X$. Then
$$\langle\vv{TxUx},\vv{xUx}\rangle\leq \langle\vv{TxUx},\vv{xTx}\rangle,$$
so
$$d^2(Tx,Ux) = \langle\vv{TxUx},\vv{TxUx}\rangle\leq 0,$$
i.e. $Tx=Ux$.
\end{proof}

One may then similarly state the corresponding definitions for jointly $(P_2)$ families of mappings. As shown in \cite{LeuNicSip18}, all those $(P_2)$ notions generalize their firmly nonexpansive counterparts and coincide with them in the case where $X$ is a Hilbert space. The main result of that paper showed that this condition suffices for the working of the proximal point algorithm, namely that if $X$ is complete, $(T_n)_{n \in \N}$ is a family of self-mappings of $X$ with a common fixed point and $(\gamma_n)_{n\in\N} \subseteq (0,\infty)$ with $\sum_{n=0}^\infty \gamma_n^2 = \infty$, then, assuming that $(T_n)$ is jointly $(P_2)$ with respect to $(\gamma_n)$, any sequence $(x_n) \subseteq X$ such that for all $n$, $x_{n+1}=T_nx_n$, is $\Delta$-convergent (a generalization of weak convergence to arbitrary metric spaces, due to Lim \cite{Lim76}) to a common fixed point of the family.

The following result attempts to elucidate the efficacy of the joint firm nonexpansiveness condition by tying it to the well-known resolvent identity.

\begin{theorem}\label{th-res-id}
Let $(T_\gamma)_{\gamma>0}$ be a family of self-mappings of $X$. Then the following are equivalent:
\begin{enumerate}[(a)]
\item
\begin{enumerate}[(i)]
\item For all $\gamma > 0$, $T_\gamma$ is nonexpansive.
\item For all $\gamma > 0$, $t \in [0,1]$ and $x \in X$,
$$T_{(1-t)\gamma}((1-t)x+tT_\gamma x) = T_\gamma x.$$
\end{enumerate}
\item $(T_\gamma)_{\gamma>0}$ is jointly firmly nonexpansive.
\end{enumerate}
\end{theorem}

\begin{proof}

``$(a)\Rightarrow(b)$'' (This is essentially expressed by \cite[Proposition 3.17]{LeuNicSip18}.) Let $\lambda$, $\mu > 0$, $x$, $y\in X$ and $\alpha$, $\beta \in [0,1]$ be such that
$(1-\alpha)\lambda = (1-\beta)\mu =:\delta.$
We get that 
\begin{align*}
d(T_{\lambda}x, T_{\mu} y) &= d(T_{(1-\alpha)\lambda}((1-\alpha)x+\alpha T_{\lambda } x), T_{(1-\beta)\mu}((1-\beta)y+\beta T_{\mu} y)) \\
&= d(T_{\delta}((1-\alpha)x+\alpha T_{\lambda} x), T_{\delta}((1-\beta)y+\beta T_{\mu} y)) \\
&\leq d((1-\alpha)x+\alpha T_{\lambda} x,(1-\beta)y+\beta T_{\mu} y).
\end{align*}

``$(b)\Rightarrow(a)$''
Let $\gamma > 0$. Take $x$, $y \in X$. To show that $d(T_\gamma x, T_\gamma y) \leq d(x,y)$, simply set in the joint firm nonexpansiveness condition $\lambda:=\gamma$, $\mu:=\gamma$, $\alpha:=0$ and $\beta:=0$.

To show that the resolvent identity holds, take $t \in [0,1]$ and $x \in X$. Set $y:=(1-t)x+ t T_\gamma x$, so one has to prove that $T_{(1-t)\gamma} y = T_\gamma x$. Then, if one sets $\lambda:=\gamma$, $\mu:=(1-t)\gamma$, $\alpha:=t$ and $\beta:=0$, since then $(1-\alpha)\lambda=(1-\beta)\mu$, one gets that
$$d(T_\gamma x, T_{(1-t)\gamma} y) \leq d((1-t)x+ t T_\gamma x, y) =0,$$
so $T_{(1-t)\gamma} y = T_\gamma x$.
\end{proof}

\section{An improvement on the uniform case}\label{sec:uniform}

If $T$ is a self-mapping of $X$ and $C$ is a nonempty subset of $X$ such that $T(C) \subseteq C$, we say that $T$ is {\it uniformly firmly nonexpansive} on $C$ with modulus $\vp : (0,\infty) \to (0,\infty)$ if for all $\eps> 0$ and $x$, $y \in C$ with $d(Tx,Ty) \geq \eps$ and all $t \in [0,1]$,
$$d^2(Tx,Ty) \leq d^2((1-t)x+tTx,(1-t)y+tTy) - 2(1-t)\vp(\eps).$$
One says that $T$ is {\it uniformly $(P_2)$} on $C$ with modulus $\vp$ if for all $\eps> 0$ and $x$, $y \in C$ with $d(Tx,Ty) \geq \eps$,
$$2d^2(Tx,Ty) \leq d^2(x,Ty) + d^2(y,Tx) - d^2(x,Tx) - d^2(y,Ty) - 2\vp(\eps),$$
or, equivalently, using the quasi-linearization function,
$$\langle \vv{TxTy},\vv{yTy}\rangle \leq \langle \vv{TxTy} ,\vv{xTx}\rangle - \vp(\eps).$$
(As before, the latter notion is more general than the first, though identical in Hilbert spaces.)

The second main result of the paper \cite{LeuNicSip18} had been the extraction of a rate of convergence for the proximal point algorithm in connection with this uniformity. There, a more restrictive notion of uniform firm nonexpansiveness had been considered  (first used in \cite[Section 3.4]{BarBauMofWan16}), where the modulus $\vp$ is nondecreasing and applied directly to the quantity $d(Tx,Ty)$, as that is the kind of modulus that is commonly used in defining the corresponding uniform notions of monotone operators and convex functions. It is easy, though, to redefine those notions in terms of this more general kind of modulus and to show that their resolvents fit into the definition above. One could also adapt \cite[Lemma 4.4]{LeuNicSip18} and \cite[Theorem 5.1]{LeuNicSip18} to this setting, but we shall not attempt that, as those results will be shortly superseded by the results of this section.

Specifically, the following generalizes \cite[Lemma 4.4]{LeuNicSip18}. (The estimate obtained below had been implicit in the original proof, but its importance had not been apparent at the time.)

\begin{lemma}\label{uniq}
Let $T$ be a self-mapping of $X$ and $C$ be a nonempty subset of $X$ such that $T(C) \subseteq C$. Assume that $T$ is uniformly $(P_2)$ on $C$ with modulus $\vp$. Then for all $x \in C$ and $z \in C \cap \Fix(T)$ with $d(Tx,z) \geq \eps$,
$$ \vp(\eps) \leq \langle \vv{Txz} ,\vv{xTx}\rangle.$$
\end{lemma}

\begin{proof}
Since $T$ is uniformly $(P_2)$ on $C$ with modulus $\vp$,
$$\langle \vv{Txz},\vv{zz}\rangle \leq \langle \vv{Txz} ,\vv{xTx}\rangle - \vp(\eps).$$
so $ \vp(\eps) \leq \langle  \vv{Txz} ,\vv{xTx} \rangle$.
\end{proof}

The following is a particular case of a recent result of Kohlenbach and Powell \cite[Lemma 3.4]{KohPowXX}, which is a quantitative version of an argument going back to e.g. \cite[Lemma 2.2]{AlbReiSho02}. We give its proof for completeness.

\begin{lemma}\label{kp}
Let $(\gamma_n)_{n \in \N}$ and $(w_n)_{n \in \N}$ be sequences of non-negative reals and $b>0$, $\theta:(0, \infty) \to \N$ and $\vp: (0, \infty) \to (0,\infty)$ be such that:
\begin{itemize}
\item for all $n \in \N$, $w_n \leq b$;
\item for all $x>0$, $\sum_{n=0}^{\theta(x)} \gamma_n \geq x$;
\item for all $\eps > 0$ and all $n \in \N$ with $\eps < w_{n+1}$, $w_{n+1} \leq w_n - \gamma_n \vp(\eps)$.
\end{itemize}
Then for all $\eps > 0$ and all $n \geq \theta\left(\frac{b+1}{\vp(\eps)}\right)+1$, $w_n \leq \eps$.
\end{lemma}

\begin{proof}
Let $\eps> 0$. Assume first that for all $n \leq \theta\left(\frac{b+1}{\vp(\eps)}\right)$, $w_{n+1} > \eps$, so for all $n \leq \theta\left(\frac{b+1}{\vp(\eps)}\right)$, $\gamma_n \vp(\eps) \leq w_n - w_{n+1}$. We get that
$$b \geq w_0 \geq w_0 - w_{\theta\left(\frac{b+1}{\vp(\eps)}\right)+1} = \sum_{n=0}^{\theta\left(\frac{b+1}{\vp(\eps)}\right)} (w_n -w_{n+1}) \geq \vp(\eps) \sum_{n=0}^{\theta\left(\frac{b+1}{\vp(\eps)}\right)} \gamma _n \geq \vp(\eps) \cdot \frac{b+1}{\vp(\eps)} = b+1,$$
a contradiction. Thus, there is an $N \leq \theta\left(\frac{b+1}{\vp(\eps)}\right)+1$ such that $w_N \leq \eps$.

We now prove by induction that for all $n \geq N$, $w_n \leq \eps$, from which it will follow that this holds for all $n \geq \theta\left(\frac{b+1}{\vp(\eps)}\right)+1$. Assume that for such an $n$, $w_n \leq \eps$ and $w_{n+1} > \eps$. But then
$$\eps < w_{n+1} \leq w_n - \gamma_n \vp(\eps) \leq w_n \leq \eps,$$
a contradiction.
\end{proof}

The general version of the above lemma has been used in \cite{KohPowXX} to give rates of convergence for iterations associated to uniformly accretive operators in uniformly convex Banach spaces (see the end of this section), following earlier work that yielded rates of convergence for the gradient flow associated to such operators in \cite{KohKou15}, and for the asymptotic behaviour towards infinity of their resolvents and the fixed-step-size proximal point algorithm in \cite{Kou17}. The general notion underlying the strategy is that of a {\it modulus of uniqueness}, which was first introduced by Kohlenbach in the early 1990s \cite{Koh90,Koh93a,Koh93b} in the context of best approximation theory, while its significance in deriving rates of convergence for asymptotically regular iterative sequences was identified later in \cite[Section 4.1]{KohOli03}.

These ideas have indeed been previously applied to our abstract proximal point algorithm in the form of \cite[Theorem 5.1]{LeuNicSip18}, but we will now show how the use of the above lemma allows us to derive the following more powerful version, where we no longer need to assume that the sequence $(d(x_n,x_{n+1})/\gamma_n)$ is nonincreasing, and we may replace the divergence of $\sum_{n=0}^{\infty} \gamma_n^2$ with that of $\sum_{n=0}^{\infty} \gamma_n$.

\begin{theorem}
Let $(T_n)_{n \in \N}$ be a family of self-mappings of $X$ and $p \in \bigcap_{n \in \N} \Fix(T_n)$. Let $(\gamma_n)_{n \in \N} \subseteq (0,\infty)$ and $\theta :(0,\infty) \to \N$ and assume that for all $x>0$,
$$\sum_{n=0}^{\theta(x)} \gamma_n \geq x.$$
Let $b > 0$ and denote by $C$ the closed ball of center $p$ and radius $b$. Let $\vp : (0, \infty) \to (0,\infty)$ and assume that for all $n$, $T_n(C) \subseteq C$ and $T_n$ is uniformly $(P_2)$ on $C$ with modulus $\gamma_n\vp$. Let $(x_n) \subseteq C$ be such that for all $n$, $x_{n+1}=T_nx_n$.

Then for all $\eps > 0$ and all $n \geq \theta\left(\frac{(b+1)^2}{\vp(\eps)}\right)+1$, $d(x_n,p) \leq \eps$. 
\end{theorem}

\begin{proof}
We seek to apply Lemma~\ref{kp} with the sequence $(d(x_n,p))$ playing the role of $(w_n)$ and with $\vp/(b+1)$ playing the role of $\vp$. We only need to show that for all $\eps>0$ and $n \in \N$ with $d(x_{n+1},p) > \eps$, we have that
$$d(x_{n+1},p) \leq d(x_n,p) - \frac{\gamma_n \vp(\eps)}{b+1}.$$
Let, then, $\eps>0$ and $n \in \N$ be such that $d(x_{n+1},p) > \eps$. Since $T_n$ is uniformly $(P_2)$ on $C$ with modulus $\gamma_n\vp$ and $x_{n+1}=T_nx_n$, we get by Lemma~\ref{uniq} that
$$\gamma_n\vp(\eps) \leq \langle \vv{x_{n+1}p},\vv{x_nx_{n+1}}  \rangle,$$
and thus we have that
$$d^2(x_{n+1},p) = \langle \vv{x_{n+1}p},\vv{x_{n+1}p} \rangle = \langle \vv{x_{n+1}p},\vv{x_np} \rangle - \langle \vv{x_{n+1}p},\vv{x_nx_{n+1}}\rangle \leq d(x_{n+1},p)d(x_n,p) - \gamma_n\vp(\eps).$$
Since $d(x_{n+1},p) > \eps > 0$, we can divide by it, so
$$d(x_{n+1},p) \leq d(x_n,p) - \frac{\gamma_n\vp(\eps)}{d(x_{n+1},p)} \leq d(x_n,p) - \frac{\gamma_n\vp(\eps)}{b+1},$$
which was what we needed to show.
\end{proof}

We see that the same phenomenon that occurred in \cite[Theorem 5.1]{LeuNicSip18} also occurs here (and already did in \cite[Lemma 3.4]{KohPowXX}) -- namely, the resulting rate of convergence does not depend in any way on the sequence of step-sizes $(\gamma_n)$, even though they are not said to belong to a compact interval in which case the fact would be explainable by general logical metatheorems.

Let us briefly remark a bit about the context in which the quantitative lemma was first introduced. One of its applications, \cite[Theorem 4.2]{KohPowXX}, deals with sequences $(x_n)$ inside a Banach space $X$ that are associated to a uniformly accretive operator $A$ and to a divergent sequence $(\alpha_n)$ of nonnegative reals in the sense that for each $n$ there is an $u \in Ax_{n+1}$ such that
$$x_{n+1}=x_n-\alpha_nu.$$
In the case where all resolvents of the operator exist (e.g. when $X$ is uniformly convex), it is immediate that the above relation simply means that for each $n$, $x_{n+1}=J_{\alpha_nA}x_n$, i.e. $(x_n)$ is a sequence generated by the proximal point algorithm with inputs $A$ and $(\alpha_n)$.

\section{The asymptotic behaviour at infinity}\label{sec:as-beh}

We now deal with what is classically denoted by {\it resolvent convergence}, i.e. the behaviour of resolvent mappings as their order tends to infinity. In the case of Hilbert spaces \cite[Lemma 1]{Bru74}, the essential idea of the proof goes back to Minty \cite{Min63} and was later popularized by Halpern \cite{Hal67}. A logical analysis of the latter result yielding a rate of metastability was first undertaken by Kohlenbach \cite[Section 4]{Koh11}. (For the extension due to Reich \cite{Rei80} to more general Banach spaces such as $L^p$ spaces, where the proof and its corresponding analysis require vastly different techniques, a rate of metastability was recently obtained by Kohlenbach and the author in \cite{KohSipXX}.)

\begin{lemma}\label{halp}
Let $T$ and $U$ be self-mappings of $X$ and $\lambda$, $\mu>0$ with $\lambda \leq \mu$. Assume that $T$ and $U$ are $(\lambda,\mu)$-mutually $(P_2)$. Let $x \in X$. Then
$$d^2(x,Ux) \geq d^2(x,Tx) + d^2(Tx,Ux).$$
In particular, $d(x,Ux) \geq d(x,Tx)$.
\end{lemma}

\begin{proof}
If $\lambda=\mu$, then, by Proposition~\ref{t-eq-u}, $Tx=Ux$ and the conclusion immediately follows. Assume, then, that $\lambda < \mu$. Since $T$ and $U$ are $(\lambda,\mu)$-mutually $(P_2)$, we have that
$$\frac1{\mu}\langle\vv{TxUx},\vv{xUx}\rangle\leq \frac1{\lambda}\langle\vv{TxUx},\vv{xTx}\rangle,$$
so
$$\frac1{\mu}\langle\vv{TxUx},\vv{xTx}\rangle + \frac1{\mu}\langle\vv{TxUx},\vv{TxUx}\rangle\leq \frac1{\lambda}\langle\vv{TxUx},\vv{xTx}\rangle,$$
from which we get that
$$0 \leq \frac1{\mu} d^2(Tx,Ux) \leq \left( \frac1\lambda - \frac1\mu \right) \langle\vv{TxUx},\vv{xTx}\rangle.$$
Since $\lambda < \mu$, $\frac1\lambda - \frac1\mu >0$, so $\langle\vv{xTx},\vv{TxUx}\rangle \geq 0$.
Thus,
$$d^2(x,Ux) = \langle\vv{xUx},\vv{xUx}\rangle = d^2(x,Tx) + d^2(Tx,Ux) + 2\langle\vv{xTx},\vv{TxUx}\rangle \geq d^2(x,Tx) + d^2(Tx,Ux),$$
and we are done.
\end{proof}

\begin{theorem}\label{t1}
Let $(T_n)_{n \in \N}$ be a family of self-mappings of $X$ and $(\gamma_n)_{n\in\N} \subseteq (0,\infty)$ a nondecreasing sequence and assume that $(T_n)$ is jointly $(P_2)$ with respect to $(\gamma_n)$. Put $F:=\bigcap_{n \in \N} \Fix(T_n)$. Let $x\in X$ and $b>0$ and assume that for all $n$, $d(x,T_nx) \leq b$. Then:
\begin{enumerate}[(a)]
\item for all $\eps > 0$ and all $g: \N \to \N$ there is an $N \leq \wt{g}^{\left(\left\lceil\frac {b^2}{\eps^2}\right\rceil \right)}(0)$ such that for all $i$, $j \in [N,N+g(N)]$, $d(T_ix,T_jx)\leq\eps$;
\item if in addition $X$ is complete and $\lim_{n \to \infty} \gamma_n = \infty$, then $F \neq \emptyset$ and $(T_nx)$ converges to the unique point in $F$ which is closest to $x$.
\end{enumerate}
\end{theorem}

\begin{proof}
\begin{enumerate}[(a)]
\item Let $\eps>0$ and $g:\N \to \N$. For all $n$, $p \in \N$ with $n \leq p$, $\gamma_n \leq \gamma_p$, and so, by Lemma~\ref{halp}, $d(x,T_px) \geq d(x,T_nx)$. Thus, $(d^2(x,T_nx))$ is a nondecreasing sequence in $[0,b^2]$ and by Proposition~\ref{qmcp}, there is an $N \leq \wt{g}^{\left(\left\lceil\frac {b^2}{\eps^2}\right\rceil \right)}(0)$ such that for all $i$, $j \in [N,N+g(N)]$, $|d^2(x,T_ix) - d^2(x,T_jx)|\leq\eps^2$.

Let $i$, $j \in [N,N+g(N)]$. Again, by Lemma~\ref{halp},
$$d^2(T_ix,T_jx) \leq |d^2(x,T_ix) - d^2(x,T_jx)|\leq\eps^2,$$
so $d(T_ix,T_jx)\leq\eps$.
\item We have that $(T_nx)$ is metastable, hence Cauchy. Since $X$ is complete, $(T_nx)$ is convergent. Denote its limit by $p$. We have that for all $m$, $n \in \N$,
$$\frac1{\gamma_m}\langle\vv{T_nxT_mT_nx},\vv{T_nxT_mT_nx}\rangle\leq \frac1{\gamma_n}\langle\vv{T_nxT_mT_nx},\vv{xT_nx}\rangle,$$
so for all $m$, $n \in \N$,
$$\frac1{\gamma_m} d^2(T_nx,T_mT_nx) \leq \frac1{\gamma_n} d(T_nx,T_mT_nx)d(x,T_nx).$$
For all $m$, $n \in \N$, either $d(T_nx,T_mT_nx) \neq 0$, so we can divide the above by it and get that
$$ d(T_nx,T_mT_nx) \leq \frac{\gamma_m}{\gamma_n}d(x,T_nx) \leq \frac{b\gamma_m}{\gamma_n},$$
or $d(T_nx,T_mT_nx)= 0$, so clearly then $d(T_nx,T_mT_nx)\leq \frac{b\gamma_m}{\gamma_n}$. Thus, for all $m \in \N$,
$$d(p,T_mp) = \lim_{n \to \infty} d(T_nx,T_mT_nx) \leq \lim_{n \to \infty} \frac{b\gamma_m}{\gamma_n} = 0,$$
so for all $m \in \N$, $p \in \Fix(T_m)$, i.e. $p \in F$ (so $F \neq \emptyset$).

Now let $q \in F$ with $q \neq p$. Then, for all $n \in \N$, since $T_n$ is $(P_2)$ and $q \in \Fix(T_n)$,
$$2d^2(T_nx,q) \leq d^2(x,q) + d^2(T_nx,q) - d^2(x,T_nx),$$
i.e.
$$d^2(x,q) -d^2(T_nx,q) - d^2(x,T_nx) \geq 0.$$
By passing to the limit, we get that
$$d^2(x,q) - d^2(p,q) - d^2(x,p) \geq 0,$$
i.e.
$$d^2(x,q) \geq d^2(x,p) + d^2(p,q) > d^2(x,p),$$
so $p$ is the unique point in $F$ which is closest to $x$.
\end{enumerate}
\end{proof}

\begin{theorem}\label{t2}
Assume that $X$ is complete. Let $(T_\gamma)_{\gamma>0}$ be a jointly $(P_2)$ family of self-mappings of $X$. Put $F:=\bigcap_{\gamma > 0} \Fix(T_\gamma)$. Let $x\in X$, $b>0$, and $(\lambda_n)_{n \in \N} \subseteq (0, \infty)$ with $\lim_{n \to \infty} \lambda_n = \infty$ and assume that for all $n$, $d(x,T_{\lambda_n}x) \leq b$. Then $F \neq \emptyset$ and the curve $(T_\gamma x)_{\gamma>0}$ is continuous and converges to the unique point in $F$ which is closest to $x$.
\end{theorem}

\begin{proof}
For all $\gamma > 0$ there is an $n$ such that $\lambda_n \geq \gamma$, and so, by Lemma~\ref{halp}, $d(x,T_\gamma x) \leq d(x,T_{\lambda_n}x) \leq b$. In addition, by Proposition~\ref{same-fixed}, for all $(\gamma_n)_{n \in \N} \subseteq (0, \infty)$, $F=\bigcap_{n \in \N} \Fix(T_{\gamma_n})$.

To show continuity, we shall show that for any $\Gamma > 0$, the curve $(T_\gamma x)_{\gamma\geq \Gamma}$ is uniformly continuous. Let $\Gamma > 0$ and $\lambda$, $\mu \geq \Gamma$ with $\lambda \leq \mu$. Then
$$\frac1{\mu}\langle\vv{T_\lambda xT_\mu x},\vv{xT_\mu x}\rangle\leq \frac1{\lambda}\langle\vv{T_\lambda xT_\mu x},\vv{xT_\lambda x}\rangle,$$
so
$$\frac1\mu d^2(T_\lambda x,T_\mu x) \leq \frac{\mu-\lambda}{\lambda\mu} \langle\vv{T_\lambda xT_\mu x},\vv{xT_\lambda x}\rangle \leq \frac{\mu-\lambda}{\lambda\mu} \cdot d(T_\lambda x,T_\mu x)d(x,T_\lambda x) \leq \frac{\mu-\lambda}{\lambda\mu} \cdot 2b^2,$$
from which we get
$$d(T_\lambda x, T_\mu x) \leq \sqrt{\frac{\mu -\lambda}\lambda} \cdot b\sqrt{2} \leq \sqrt{\mu-\lambda}\cdot \frac{b\sqrt{2}}{\sqrt{\Gamma}},$$
which implies the sought-after uniform continuity.

By applying Theorem~\ref{t1}, we get that $F \neq \emptyset$ and that for all nondecreasing $(\gamma_n)_{n \in \N} \subseteq (0, \infty)$ such that $\lim_{n \to \infty} \gamma_n = \infty$, the sequence $(T_{\gamma_n}x)_{n \in \N}$ converges to the unique point in $F$ which is closest to $x$, so the curve $(T_\gamma x)_{\gamma>0}$ converges to that same point.
\end{proof}

We see that the above result subsumes \cite[Theorem 1]{Hal67}, \cite[Lemma 1]{Bru74}, \cite[Theorem 24.1]{GoeRei84}, \cite[Theorem 3.1.1]{Jos97} and \cite[Theorem 1.4]{BacRei14}.

\mbox{}

\mbox{}

Thus, these examples have shown, in addition to the results in \cite{LeuNicSip18}, that jointly firmly nonexpansive families of mappings unify a number of algorithms involving resolvent-type mappings, and further work in this direction may increase the area of applicability even more. A promising avenue would be to investigate how the notion may be extended to Banach spaces which are more general than Hilbert spaces, allowing the abstract treatment of results such as the ones in \cite{BruRei77, ReiSha87, ReiSab11, MarReiSab12, MarReiSab13a, MarReiSab13b, MarReiSab13c}.

\section{Acknowledgements}

This work has been supported by the German Science Foundation (DFG Project KO 1737/6-1) and by a grant of the Romanian National Authority for Scientific Research, CNCS - UEFISCDI, project number PN-III-P1-1.1-PD-2019-0396.

I would like to thank Liviu P\u aunescu for suggesting to investigate continuity in Theorem~\ref{t2}.

\end{document}